\documentclass[ECP]{ejpecp} 

\SHORTTITLE{QSD of explosive CSBP} 

\TITLE{Quasi-stationary distributions\\
associated with explosive CSBP} 

\AUTHORS{%
  Cyril~Labb\'e\footnote{LPMA, Universit\'e Pierre et Marie Curie (Paris 6), Bo\^ite 188, 4 place Jussieu, 75252 Paris Cedex 05, France.
    \EMAIL{cyril.labbe@upmc.fr}}}


\AMSSUBJ{60J80; 60F05} 

\SUBMITTED{December 18, 2012} 
\ACCEPTED{July 3, 2013} 

\VOLUME{18}
\YEAR{2013}
\PAPERNUM{57}
\DOI{v18-2508}

\ABSTRACT{We characterise all the quasi-stationary distributions and the $Q$-process associated with a continuous state branching process that explodes in finite time. We also provide a rescaling for the continuous state branching process conditioned on non-explosion when the branching mechanism is regularly varying at $0$.}

\newcommand{\cqfd}{ \hfill $\square$}
\newcommand{\tun}{\mathtt{1}}
\newcommand{\cF}{\mathcal{F}}
\newcommand{\cO}{\mathcal{O}}
\newcommand{\cZ}{\mathcal{Z}}
\newcommand{\rT}{\mathrm{T}}
\newcommand{\rZ}{\mathrm{Z}}
\newcommand{\bE}{\mathbf{E}}
\newcommand{\bP}{\mathbf{P}}

\newcommand{\bbD}{\mathbb{D}}
\newcommand{\bbE}{\mathbb{E}}
\newcommand{\bbN}{\mathbb{N}}
\newcommand{\bbP}{\mathbb{P}}
\newcommand{\bbQ}{\mathbb{Q}}

\newcommand{\bbZ}{\mathbb{Z}}

\begin{document}



\section{Introduction}
Continuous-state branching processes (CSBP) are $[0,\infty]$-valued Markov processes that describe the evolution of the size of a continuous population. They have been introduced by Jirina~\cite{Jirina58} and Lamperti~\cite{Lamperti67}. We recall some basic facts on CSBP and refer to Bingham~\cite{Bingham76}, Grey~\cite{Grey74}, Kyprianou~\cite{Kyprianou06} and Le Gall~\cite{LeGall99} for details and proofs.\\
Consider the space $\bbD([0,\infty),[0,\infty])$ of c\`adl\`ag $[0,\infty]$-valued functions endowed with the Skorohod's topology. We denote by $\rZ:=(\rZ_t,t\geq 0)$ the canonical process on this space. For all $x\in[0,\infty]$, we denote by $\bbP_x$ the distribution of the CSBP starting from $x$ whose semigroup is characterised by
\begin{equation}\label{EqLogLaplace}
\forall t\geq 0,\lambda > 0,\;\;\bbE_x[e^{-\lambda \rZ_t}]=e^{-x\, u(t,\lambda)}
\end{equation}
where for all $\lambda > 0$, $(u(t,\lambda),t\geq 0)$ is the unique solution of
\begin{equation}\label{EqDefUt}
\partial_t u(t,\lambda) = -\Psi(u(t,\lambda))\;\;,\;\; u(0,\lambda)=\lambda
\end{equation}
and $\Psi$, the so-called \textit{branching mechanism} of the CSBP, is a convex function of the form
\begin{equation}\label{EqPsi}
\forall u\geq 0,\;\;\Psi(u) = \gamma u + \frac{\sigma^2}{2}u^2 + \int_{(0,\infty)}\!\!\!(e^{-uh}-1+uh\mathbf{1}_{\{h<1\}})\,\nu(dh)
\end{equation}
where $\gamma \in \mathbb{R}$, $\sigma \geq 0$ and $\nu$ is a Borel measure on $(0,\infty)$ such that $\int_{(0,\infty)}(1\wedge h^2)\nu(dh)<\infty$. The function $\Psi$ entirely characterises the law of the process. The CSBP fulfils the following branching property: for all $x,y \in [0,\infty]$ the process starting from $x+y$ has the same law as the sum of two independent copies starting from $x$ and $y$ respectively. Observe that $\Psi$ is also the Laplace exponent of a spectrally positive L\'evy process, we refer to Theorem 1 in~\cite{Lamperti67} for a pathwise correspondence between L\'evy processes and CSBP.

The convexity of $\Psi$ entails that the ratio $\Psi(u)/u$ is increasing. A direct calculation or Proposition I.2 p.16~\cite{BertoinBookLevy96} shows that it converges to a finite limit as $u\rightarrow\infty$ iff
\begin{equation}\label{EqFiniteVar}
\textrm{(\textit{Finite variation})\hspace{5mm}} \sigma = 0\mbox{ and } \int_{(0,1)}\!\!\!h\nu(dh) < \infty
\end{equation}
When this condition is verified, the limit of the ratio is necessarily equal to $D:=\gamma + \int_{(0,1)}h\nu(dh)$ and $\Psi$ can be rewritten
\begin{equation}\label{EqPsiFV}
\forall u\geq 0,\;\;\Psi(u) = Du + \int_{(0,\infty)}\!\!\!(e^{-uh}-1)\,\nu(dh)
\end{equation}

As $t\rightarrow\infty$ the CSBP converges either to $0$ or to $\infty$, which are absorbing states for the process. Consequently we define the \textit{lifetime} of the CSBP as the stopping time $\rT:=\rT_0\wedge\rT_\infty$ where
$$ \textrm{(\textit{Extinction}) \hspace{2mm}}\rT_0 := \inf\{t\geq 0: \mathrm{Z}_t =0\}\;\;,\;\;\textrm{(\textit{Explosion}) \hspace{2mm}}\rT_{\infty} := \inf\{t\geq 0: \mathrm{Z}_t =\infty\}$$
We denote by $q:=\sup\{u\geq0:\Psi(u)\leq 0\}\in[0,\infty]$ the second root of the convex function $\Psi$: it is elementary to check from (\ref{EqDefUt}) that $u(t,q)=q$ for all $t\geq 0$ and that for all $\lambda > 0$, $u(t,\lambda)\rightarrow q$ as $t\rightarrow\infty$. Hence from (\ref{EqLogLaplace}) we get
$$ \forall x\in[0,\infty],\;\;\bbP_x\big(\lim\limits_{t\rightarrow\infty}\rZ_t=0)=1-\bbP_x\big(\lim\limits_{t\rightarrow\infty}\rZ_t=\infty)=e^{-xq}$$
When $\Psi'(0+) > 0$ (resp. $\Psi'(0+) = 0$) the CSBP is said \textit{subcritical} (resp. \textit{critical}), the convexity of $\Psi$ then implies $q=0$ and the process is almost surely absorbed at $0$. Moreover the extinction time $\rT_0$ is almost surely finite iff
\begin{equation}\label{EqExtin}
\int^{+\infty} \frac{du}{\Psi(u)} < \infty
\end{equation}
Otherwise $\rT_0$ is almost surely infinite. When $\Psi'(0+) \in [-\infty,0)$ the CSBP is said \textit{supercritical} and then $q\in(0,\infty]$. The CSBP has a positive probability to be absorbed at $0$ iff $q\in(0,\infty)$. In that case, on the extinction event $\{\rT=\rT_0\}$ the finiteness of $\rT_0$ is governed by the same criterion as above. On the explosion event $\{\rT=\rT_\infty\}$, the explosion time $\rT_\infty$ is almost surely finite iff
\begin{equation}\label{EqExplo}
\int_{0+}\frac{du}{-\Psi(u)} < \infty
\end{equation}
Observe that $\Psi'(0+)=-\infty$ is required (but not sufficient) for this inequality to be fulfilled. When (\ref{EqExplo}) does not hold, $\rT_\infty$ is almost surely infinite on the explosion event.
\medskip

By quasi-stationary distribution (QSD for short), we mean a probability measure $\mu$ on $(0,\infty)$ such that
\begin{equation*}
\mathbb{P}_{\mu}(\mathrm{Z}_t \in \cdot \, |\, \rT > t) = \mu(\cdot)
\end{equation*}
When $\mu$ is a QSD, it is a simple matter to check that under $\mathbb{P}_{\mu}$ the random variable $\rT$ has an exponential distribution, the parameter of which is called the \textit{rate of decay} of $\mu$. The goal of the present paper is to investigate the QSD associated with a CSBP that explodes in finite time almost surely.

\subsection{A brief review of the literature: the extinction case}
Li~\cite{Li00} and Lambert~\cite{Lambert07} considered the extinction case $\rT=\rT_0 < \infty$ almost surely, so that $\Psi'(0+) \geq 0$ and (\ref{EqExtin}) holds, and they studied the CSBP conditioned on non-extinction. We recall some of their results. When $\Psi$ is subcritical, that is $\Psi'(0+) > 0$, there exists a family $(\mu_{\mbox{\tiny$\beta$}};0 < \beta \leq \Psi'(0+))$ of QSD where $\beta$ is the rate of decay of $\mu_{\mbox{\tiny$\beta$}}$. These distributions are characterised by their Laplace transforms as follows
\begin{equation}\label{EqQSDExtinction}
\forall \lambda \geq 0,\;\;\int_{(0,\infty)}\mu_{\mbox{\tiny$\beta$}}(dr)e^{-r\lambda}= 1-e^{-\beta\Phi(\lambda)}\;\;\;\;\mbox{where}\;\;\;\;\Phi(\lambda):=\int_{\lambda}^{+\infty}\frac{du}{\Psi(u)}
\end{equation}
Notice that $\Phi$ is well-defined thanks to (\ref{EqExtin}). For any $\beta > \Psi'(0+)$ they proved that there is no QSD with rate of decay $\beta$, and that Equation (\ref{EqQSDExtinction}) does not define the Laplace transform of a probability measure on $(0,\infty)$. Additionally, the value $\beta=\Psi'(0+)$ yields the so-called Yaglom limit:
\begin{equation*}
\forall x > 0,\;\;\mathbb{P}_{x}(\mathrm{Z}_t \in \cdot\, |\, \rT > t) \underset{t\rightarrow\infty}{\longrightarrow} \mu_{\mbox{\tiny$\Psi'(0+)$}}(\cdot)
\end{equation*}
When $\Psi$ is critical, that is $\Psi'(0+)=0$, the preceding quantity converges to a trivial limit for all $x>0$ and Equation (\ref{EqQSDExtinction}) does not define the Laplace transform of a probability measure on $(0,\infty)$. However, under the condition $\Psi''(0+) <\infty$, they proved the following convergence (that extends a result originally due to Yaglom~\cite{Yaglom47} for Galton-Watson processes)
\begin{equation}\label{CriticalLiLambert}
\forall x > 0,z\geq 0,\;\;\mathbb{P}_x\Big(\frac{\mathrm{Z}_t}{t}\geq z\,\big|\,\rT > t\Big) \underset{t\rightarrow\infty}{\longrightarrow} \exp\Big(-\frac{2z}{\Psi''(0+)}\Big)
\end{equation}
Finally in both critical and subcritical cases, for any given value $t > 0$ the process $(\rZ_r,r \in [0,t])$ conditioned on $s < \rT$ admits a limiting distribution as $s\rightarrow\infty$, called the $Q$-process. The law of the $Q$-process is obtained as a $h$-transform of $\mathbb{P}$ as follows
\begin{equation*}
\forall x >0,\;\;d\mathbb{Q}_{x|{\cal F}_t}:=\frac{\mathrm{Z}_t \, e^{\Psi'(0)t}}{x} \, d\mathbb{P}_{x|{\cal F}_t}
\end{equation*}

\subsection{Main results: the explosive case}
We now assume that almost surely the CSBP explodes in finite time. From the results recalled above, this is equivalent with (\ref{EqExplo}) and $q=\infty$ so that $\Psi$ is convex, decreasing and non-positive. Hence the ratio $\Psi(u)/u$ cannot converge to $+\infty$ so that necessarily (\ref{EqFiniteVar}) holds, and $\Psi$ can be written as in (\ref{EqPsiFV}). Observe also that in that case the L\'evy process with Laplace exponent $\Psi$ is a subordinator. We set:
\begin{equation*}
\Psi(+\infty) := \lim\limits_{u\rightarrow\infty}\Psi(u) \in [-\infty,0)
\end{equation*}
From (\ref{EqPsiFV}) we deduce that $\Psi(+\infty)\in(-\infty,0)$ iff $\nu(0,\infty)<\infty$ and $D=0$. When this condition holds, we have $\Psi(+\infty)=-\nu(0,\infty)$. Otherwise $\Psi(+\infty)=-\infty$.

We start with an elementary remark: conditioning a CSBP on non-explosion does not affect the branching property. Consequently the law of $\rZ_t$ conditioned on $\rT>t$ is infinitely divisible: if it admits a limit as $t$ goes to $\infty$, the limit has to be infinitely divisible as well. Our result below shows that $\Psi(+\infty)$ plays a r\^ole analogue to $\Psi'(0+)$ in the extinction case.
\begin{theorem}\label{ThQSD}
Suppose $\rT=\rT_{\infty}<\infty$ almost surely and set
\begin{equation*}
\forall \lambda \geq 0,\;\;\Phi(\lambda):=\int_{\lambda}^{0}\frac{du}{\Psi(u)}
\end{equation*}
For any $\beta >0$ there exists a unique quasi-stationary distribution $\mu_{\mbox{\tiny$\beta$}}$ associated to the rate of decay $\beta$. This probability measure is infinitely divisible and is characterised by
\begin{equation}\label{EqQSD}
\forall \lambda\geq 0,\;\;\int_{(0,\infty)}\mu_{\mbox{\tiny$\beta$}}(dr)e^{-r\lambda}= e^{-\beta\Phi(\lambda)}
\end{equation}
Additionally, the following dichotomy holds true:
\begin{enumerate}
\item[{\rm(i)}] $\Psi(+\infty) \in (-\infty,0)$. The limiting conditional distribution is given by
\begin{equation*}
\forall x \in (0,\infty),\;\;\lim\limits_{t\rightarrow\infty}\mathbb{P}_x(\mathrm{Z}_t\in \cdot \, |\,\rT>t)=\mu_{\mbox{\tiny$x\nu(0,\infty)$}}(\cdot)
\end{equation*}
\item[{\rm(ii)}] $\Psi(+\infty)=-\infty$. The limiting conditional distribution is trivial:
\begin{equation*}
\forall a,x \in (0,\infty),\;\;\lim\limits_{t\rightarrow\infty}\mathbb{P}_x(\mathrm{Z}_t \leq a \, |\,\rT>t)=0
\end{equation*}
\end{enumerate}
\end{theorem}
\noindent Let us make some comments. Firstly this theorem implies that $\lambda\mapsto\Phi(\lambda)$ is the Laplace exponent of a subordinator, and so, $\mu_{\mbox{\tiny$\beta$}}$ is the distribution of a $\Phi$-L\'evy process taken at time $\beta$. Secondly there is a similarity with the extinction case: the limiting conditional distribution is trivial iff $\Psi(+\infty)=-\infty$ so that the dichotomy on the value $\Psi(+\infty)$ is the explosive counterpart of the dichotomy on the value $\Psi'(0+)$ in the extinction case. Also, note the similarity in the definition of the Laplace transforms (\ref{EqQSDExtinction}) and (\ref{EqQSD}). However, there are two major differences with the extinction case: firstly there is no restriction on the rates of decay. Secondly, even if the limiting conditional distribution is trivial when $\Psi(+\infty)=-\infty$, there exists a family of QSD.

The following theorem characterises the $Q$-process associated with an explosive CSBP. Let ${\cal F}_t$ be the sigma-field generated by $(\rZ_r,r\in[0,t])$, for any $t\in[0,\infty)$.
\begin{theorem}\label{ThQDistrib}
We assume that $\rT=\rT_\infty < \infty$ almost surely. For each $x > 0$, there exists a distribution $\bbQ_x$ on $\bbD([0,\infty),[0,\infty))$ such that for any $t\geq 0$
\begin{equation*}
\lim\limits_{s\rightarrow\infty}\mathbb{P}_x(\cdot \,| \,\rT>s)_{|\cF_t} = \mathbb{Q}_x(\cdot)_{|\cF_t}
\end{equation*}
Furthermore, $\mathbb{Q}_x$ is the law of the $\Psi^{Q}$-CSBP where
\begin{equation*}
\Psi^{Q}(u) = Du
\end{equation*}
\end{theorem}
\noindent The $Q$-process appears as the $\Psi$-CSBP from which one has removed all the jumps: only the deterministic part remains, see also the forthcoming Proposition \ref{PropDt}. Notice that the $Q$-process cannot be defined through a $h$-transform of the CSBP: actually the distribution of the $Q$-process on $\bbD([0,t],[0,\infty))$ is not even absolutely continuous with respect to that of the $\Psi$-CSBP, except when the L\'evy measure $\nu$ is finite.

\medskip

When $\Psi(+\infty)=-\infty$, Theorem \ref{ThQSD} shows that the process conditioned on non-explosion converges to a trivial limit. In the next theorem, under the assumption that the branching mechanism is regularly varying at $0$ we propose a rescaling of the CSBP conditioned on non-explosion such that it converges to a non-trivial limit. Recall that we call slowly varying function at $0$ any continuous map $L:(0,\infty)\rightarrow(0,\infty)$ such that for any $a \in (0,\infty)$, $L(au)/L(u)\rightarrow 1$ as $u\downarrow 0$.
\begin{theorem}\label{ThRegularly}
Suppose that $\Psi(u)=-u^{1-\alpha}L(u)$ with $L$ a slowly varying function at $0$ and $\alpha \in (0,1)$, and assume that $\Psi(+\infty)=-\infty$. Consider any function $f:[0,\infty)\rightarrow(0,\infty)$ satisfying $\Psi\big(f(t)^{-1}\big)f(t) \sim \Psi(u(t,0+))$ as $t\rightarrow\infty$. Then the following convergence holds true:
\begin{equation*}
\forall x,\lambda \in(0,\infty),\;\;\bbE_x\Big[e^{-\lambda \rZ_t / f(t)} \, \big| \, t < \rT\Big] \underset{t\rightarrow\infty}{\longrightarrow} e^{-x\,\lambda^{\alpha} / \alpha}
\end{equation*}
\end{theorem}
\noindent Observe that the limit displayed by this theorem is the Laplace transform of the QSD associated with $\Psi(u)=-u^{1-\alpha}$.
\begin{example}
When $\Psi(u)=-k\,u^{1-\alpha}$ with $k > 0$ and $\alpha \in (0,1)$, we have $f(t) \sim (\alpha k t)^{(1-\alpha)/\alpha^2}$ as $t\rightarrow\infty$. When $\Psi(u)=-c\, u -k\, u^{1-\alpha}$ with $k,c > 0$ and $\alpha \in (0,1)$, we have $f(t) \sim (k/c)^{(1-\alpha)/\alpha^2}e^{ct/\alpha}$ as $t\rightarrow\infty$.
\end{example}

The proof of Theorem \ref{ThRegularly} is inspired by calculations of Slack in~\cite{Slack68} where it is shown that any critical Galton-Watson process with a regularly varying generating function can be properly rescaled so that, conditioned on non-extinction, it converges towards a non-trivial limit. For completeness we also adapt the result of Slack to critical CSBP conditioned on non-extinction.
\begin{proposition}\label{PropCritical}
Suppose that $\Psi(u)=u^{1+\alpha}L(u)$ with $L$ a slowly varying function at $0$ and $\alpha\in(0,1]$. Assume that $\rT=\rT_0 <\infty$ almost surely. Fix any function $f:[0,\infty) \rightarrow (0,\infty)$ verifying $f(t) \sim u(t,\infty)$ as $t\rightarrow\infty$. Then we have the following convergence
\begin{equation*}
\forall x,\lambda \in (0,\infty),\;\;\bbE_x\big[e^{-\lambda\, \rZ_t f(t)}\, | \, t < \rT\big] \underset{t\rightarrow\infty}{\longrightarrow} 1 - \big(1+\lambda^{-\alpha}\big)^{-1/\alpha}
\end{equation*}
\end{proposition}
\noindent We recover in particular the finite variance case (\ref{CriticalLiLambert}) of Lambert and Li. Our result also covers the so-called stable branching mechanisms $\Psi(u)=u^{1+\alpha}$ with $\alpha\in (0,1]$.
\paragraph{Organisation of the paper.} We start with a study of continuous-time Galton-Watson processes (which are the discrete-state counterparts of CSBP): we provide a complete description of the QSD when this process explodes in finite time almost surely and compare the results with the continuous-state case. In the third section we prove Theorems \ref{ThQSD}, \ref{ThQDistrib} and \ref{ThRegularly}. Finally in the fourth section we prove Proposition \ref{PropCritical}.

\section{The discrete case}
A discrete-state branching process $(\cZ_t,t\geq 0)$ is a continuous-time Markov process taking values in $\mathbb{Z}_+\cup\{+\infty\}$ that verifies the branching property (we refer to Chapter V of Harris~\cite{Harris63} for the proofs of the following facts). It can be seen as a Galton-Watson process with offspring distribution $\xi$ where each individual has an independent exponential lifetime with parameter $c > 0$. Let us denote by $\phi(\lambda) =\sum_{k=0}^{\infty}\lambda^k\xi(k),\, \forall\lambda \in [0,1]$ the generating function of the Galton-Watson process. We denote by $\bP_n$ the law on the space $\bbD([0,\infty),\bbZ_+\cup\{+\infty\})$ of $\cZ$ starting from $n\in\bbZ_+\cup\{+\infty\}$, and $\bE_n$ the related expectation operator. The semigroup of the DSBP is characterised via the Laplace transform (see Chapter V.4 of~\cite{Harris63})
\begin{equation}\label{EqF}
\forall r \in (0,1),\forall t\in[0,\infty),\;\;\bE_n\big[\, r^{\cZ_t}\, \big]=F(t,r)^{n}\;\mbox{where}\;\int_r^{F(t,r)}\!\!\!\frac{dx}{c\, (\phi(x)-x)}=t
\end{equation}
Let $\tau$ be the lifetime of $\cZ$, that is, the infimum of the extinction time $\tau_0$ and the explosion time $\tau_{\infty}$. Taking the limits $r\downarrow 0$ and $r\uparrow 1$ in (\ref{EqF}) one gets
$$ \bP_n(\tau_0 \leq t) = F(t,0+)^n\;\;,\;\;\bP_n(\tau_{\infty} < t) = 1-F(t,1-)^n$$
In this section, we assume that there is explosion in finite time almost surely. Results of Chapters V.9 and V.10 of~\cite{Harris63} then entail that the smallest solution of the equation $\phi(x)=x$ equals $0$ (and so $\xi(0)=0$) and that $\int_{1-}\frac{dx}{c\, (\phi(x)-x)}$ is finite. This allows to define
\begin{equation}\label{EqPhiDSBP}
\Phi(r) := \int_1^r\frac{dx}{c\, (\phi(x)-x)},\; r\in (0,1]
\end{equation}
Clearly $r\mapsto\Phi(r)$ is the inverse map of $t\mapsto F(t,1-)$, that is for all $t \geq 0, \Phi\big(F(t,1-)\big)=t$. 
We say that a measure $\mu$ on $\mathbb{N}=\{1,2,\ldots\}$ is a quasi-stationary distribution (QSD) for $\cZ$ if
\begin{equation*}
\bP_{\mu}(\cZ_t\in\cdot\, | \,\tau > t) = \mu(\cdot)
\end{equation*}
From the Markov property, we deduce that $\tau$ has an exponential distribution under $\bP_{\mu}$, the parameter of which is called the rate of decay of $\mu$.
\begin{theorem}\label{ThDSBP}
Suppose there is explosion in finite time almost surely. Let $\beta_0 := c\, (1-\xi(1))$. There is a unique quasi-stationary distribution $\mu_{\mbox{\tiny$\beta$}}$ associated with the rate of decay $\beta$ if and only if $\beta$ is of the form $n\beta_0$, with $n\in\mathbb{N}$. It is characterised by its Laplace transform
\begin{equation}\label{EqLaplaceDSBP}
\sum_k\mu_{\mbox{\tiny$\beta$}}(\{k\})r^k = e^{-\beta\Phi(r)},\ \forall r \in (0,1]
\end{equation}
For any initial condition $n\in \mathbb{N}$ we have
\begin{equation*}
\lim\limits_{t\rightarrow\infty}\bP_n(\, \cZ_t\in\cdot \, | \, \tau>t \, ) = \mu_{\mbox{\tiny$n\beta_0$}}(\cdot)
\end{equation*}
\end{theorem}
Let us make some comments. First there exists only a countable family of QSD. This is due to the restrictive condition that our process takes values in $\bbZ_+\cup\{\infty\}$. Also, observe the similarity with Theorem \ref{ThQSD}: indeed a DSBP can be seen as a particular CSBP starting from an integer and whose branching mechanism is the Laplace exponent of a compound Poisson process with integer-valued jumps. In particular $\nu(\{k\}) = c\,\xi(k+1)$ for all integer $k\geq 1$. Hence the quantity $c(1-\xi(1))$ in the DSBP case corresponds to $\nu(0,\infty)$ in the CSBP case. Finally we mention that the $Q$-process associated with an explosive DSBP is the constant process, that is, the DSBP with the trivial generating function $F(t,r)=r$. This fact can be proved using calculations similar to those in the proof below or it can be deduced from Theorem \ref{ThQDistrib} and the remarks above.
\begin{proof}
We start with the proof of the uniqueness of the QSD for a given rate of decay $\beta > 0$. Let $\mu$ be a QSD and let $\beta > 0$ be its rate of decay. Then we have for all $t\geq 0$
\begin{equation*}
e^{-\beta t}=\mathbb{P}_{\mu}(\tau>t)=\sum_k\mu(\{k\})\mathbb{P}_{k}(\tau>t)=\sum_k\mu(\{k\})F(t,1-)^k
\end{equation*}
Since $F(\Phi(r),1-)=r$ we get
\begin{equation*}
\forall r \in (0,1],\;\;e^{-\beta\Phi(r)}=\sum_k\mu(\{k\})r^k
\end{equation*}
which ensures the uniqueness of the QSD for a given rate of decay. We now prove that whenever $\beta\!=\!n\beta_0$ with $n\in\bbN$, the last expression is indeed the Laplace transform of a probability measure on $\bbN$.
\begin{equation}\label{EqCVDSBP}
\forall n\in\bbN,\;\;\mathbb{E}_n[r^{\mathrm{Z}_t}|\tau>t] = \frac{\mathbb{E}_n[r^{\mathrm{Z}_t};\tau>t]}{\mathbb{P}_n(\tau>t)}= \bigg(\frac{F(t,r)}{F(t,1-)}\bigg)^n
\end{equation}
By $0 \leq F(t,r)\leq F(t,1-)\rightarrow 0$ as $t\rightarrow\infty$, $\phi(x)=\xi(1)x+\cO(x^2)$ as $x\downarrow 0$ and (\ref{EqPhiDSBP}) we get
\begin{eqnarray*}
\Phi(r)=\int_{F(t,1-)}^{F(t,r)}\frac{dx}{c(\phi(x)-x)}&\underset{t\rightarrow\infty}{\sim}&\int_{F(t,1-)}^{F(t,r)}\frac{dx}{cx(\xi(1)-1)} = -\frac{1}{\beta_0}\log\frac{F(t,r)}{F(t,1-)}
\end{eqnarray*}
We deduce that the r.h.s. of (\ref{EqCVDSBP}) converges to $\exp(-\Phi(r)n\beta_0)$ as $t\rightarrow\infty$. From this convergence and the fact that $\Phi(1-)=0$, we deduce that $r\mapsto\exp(-\Phi(r)n\beta_0)$ is the Laplace transform of a probability measure say $\mu_{\mbox{\tiny$n\beta_0$}}$ on $\mathbb{Z}_+$. As $\Phi(0+)= +\infty$, we deduce that this probability measure does not charge $0$. Also, observe that $\mu_{\mbox{\tiny$\beta_0$}}(\{1\}) > 0$. Indeed for all $r\in(0,1)$ we have $\Phi'(r)=-(\beta_0 r)^{-1} - G(r)$ where $G$ is bounded near $0$. Since $\mu_{\mbox{\tiny$\beta_0$}}(\{1\})=-\lim_{r\downarrow 0}\beta_0\Phi'(r)e^{-\beta_0\Phi(r)}$, the strict positivity follows.\vspace{3pt}\\
Fix $\beta > 0$. We now assume that $r\mapsto e^{-\beta\Phi(r)}$ is the Laplace transform of a probability measure on $\bbN$ say $\mu_{\mbox{\tiny$\beta$}}$. Denote by $m \in \mathbb{N}$ the smallest integer such that $\mu_{\mbox{\tiny$\beta$}}(\{m\}) > 0$. Then we have for all $r\in(0,1]$
\begin{eqnarray*}
e^{-\beta\Phi(r)}&=&\mu_{\mbox{\tiny$\beta$}}(\{m\})r^m + \sum_{k>m}\mu_{\mbox{\tiny$\beta$}}(\{k\})r^k=(e^{-\beta_0\Phi(r)})^{\frac{\beta}{\beta_0}}\\
&=&\Big(\mu_{\mbox{\tiny$\beta_0$}}(\{1\})r + \sum_{k>1}\mu_{\mbox{\tiny$\beta_0$}}(\{k\})r^k\Big)^{\frac{\beta}{\beta_0}}
\end{eqnarray*}
This implies that $\mu_{\mbox{\tiny$\beta$}}(\{m\})r^m \sim (\mu_{\mbox{\tiny$\beta_0$}}(\{1\})r)^{\frac{\beta}{\beta_0}}$ as $r\downarrow 0$ and so, $m=\frac{\beta}{\beta_0}\in \mathbb{N}$. Consequently (\ref{EqLaplaceDSBP}) is the Laplace transform of a probability measure on $\bbN$ iff $\beta$ is of the form $n\beta_0$.
\end{proof}

\section{Quasi-stationary distributions and Q-process in the explosive case}
Consider a branching mechanism $\Psi$ of the form (\ref{EqPsi}). It is well-known and can be easily checked from (\ref{EqLogLaplace}) that for any $t\geq 0$ the law of $\rZ_t$ under $\bbP_x$ is infinitely divisible. Consequently $u(t,\cdot)$ is the Laplace exponent of a (possibly killed) subordinator (see Chapter 5.1~\cite{Kyprianou06}). Thanks to the L\'evy-Khintchine formula, there exist $a_t,d_t \geq 0$ and a Borel measure $w_t$ on $(0,\infty)$ with $\int_{(0,\infty)}(1\wedge h)w_t(dh)<\infty$ such that
\begin{equation}\label{EqLKUt}
\forall \lambda\geq 0,\;\;u(t,\lambda) = a_t + d_t\lambda + \int_{(0,\infty)}\!\!\!(1-e^{-\lambda h})\,w_t(dh)
\end{equation}
Note that $a_t\!=\!u(t,0+)$ is positive iff the CSBP has a positive probability to explode in finite time. In the genealogical interpretation, the measure $w_t$ gives the distribution of the clusters of individuals alive at time $t$ who share a same ancestor at time $0$, while the coefficient $d_t$ corresponds to the individuals at time $t$ who do not share their ancestor at time $0$ with other individuals. For further use, we write the integral version of (\ref{EqDefUt}):
\begin{equation}\label{EqUt}
\forall t\geq 0, \forall \lambda \in[0,\infty)\backslash\{q\},\;\;\int_{u(t,\lambda)}^\lambda \frac{du}{\Psi(u)}=t
\end{equation}
The following result shows that the drift $d_t$ is left unchanged when replacing $\Psi$ by $\Psi^Q$ of Theorem \ref{ThQDistrib}: this means that the $Q$-process is obtained by removing all the clusters in the population.
\begin{proposition}\label{PropDt}
When $\Psi$ fulfils (\ref{EqFiniteVar}) then $d_t=e^{-Dt}$ for all $t\geq 0$. Otherwise $d_t=0$ for all $t>0$.
\end{proposition}
\begin{proof}
Corollary p.1049 in~\cite{Silverstein68} entails that $d_t=0$ for all $t>0$ whenever $\sigma>0$ or $\int_{(0,1)}h\nu(dh)=\infty$. We now assume the converse, namely that $\Psi$ fulfils (\ref{EqFiniteVar}) so that $\Psi(u)/u \rightarrow D$ as $u\rightarrow \infty$. A direct computation shows that $d_t = \lim_{\lambda\rightarrow\infty}u(t,\lambda)/\lambda$. Then for any $t\geq 0, \lambda > 0$
\begin{equation}\label{EqCVDt}
\log\Big(\frac{u(t,\lambda)}{\lambda}\Big) = \int_0^t\frac{\partial_su(s,\lambda)}{u(s,\lambda)}ds = -\int_0^t\frac{\Psi(u(s,\lambda))}{u(s,\lambda)}ds
\end{equation}
If $q\in(0,\infty)$, then for all $\lambda > q$ and all $0\! \leq\! s \!\leq\! t$ we have $q < u(t,\lambda) \leq u(s,\lambda) \leq \lambda$ thanks to (\ref{EqDefUt}) and by (\ref{EqUt}) we deduce that $u(t,\lambda)\uparrow\infty$ as $\lambda\rightarrow\infty$. If $q=\infty$, then for all $\lambda > 0$ and all $0 \!\leq\! s \!\leq\! t$ we have $\lambda \leq u(s,\lambda) \leq u(t,\lambda)$ thanks to (\ref{EqDefUt}) and obviously $u(t,\lambda)\uparrow\infty$ as $\lambda\rightarrow\infty$. Since $\Psi(u)/u \uparrow D$ as $u\rightarrow\infty$ the dominated convergence theorem applied to (\ref{EqCVDt}) yields that $\log(u(t,\lambda)/\lambda) \rightarrow -Dt$ as $\lambda\rightarrow\infty$.
\end{proof}

Until the end of the section, we assume that $\Psi$ verifies (\ref{EqExplo}) and that $q=\infty$. Consequently under $\bbP_x$, $\rZ$ explodes in finite time almost surely and $a_t=u(t,0+)>0$ for all $t>0$. An elementary calculation entails
\begin{equation*}
\forall t\geq 0, x > 0,\;\;\mathbb{P}_x(\rT>t)=e^{-x\, a_t}
\end{equation*}
We introduce for all $\lambda\geq 0$, $\Phi(\lambda):= \int_\lambda^0 du/\Psi(u)$. This non-negative, increasing function admits a continuous inverse, namely the function $t\mapsto a_t$. Also, thanks to Equation (\ref{EqUt}) we deduce the identities
\begin{equation}\label{EqPhiU}
\forall t,\lambda \geq 0,\;\;\Phi(u(t,\lambda)) = t + \Phi(\lambda)\;,\;u(t,\lambda) = u(t+\Phi(\lambda),0+)
\end{equation}
\subsection{Proof of Theorem \ref{ThQSD}}
First we compute the necessary form of the QSD. Fix $\beta > 0$ and suppose that $\mu_{\mbox{\tiny$\beta$}}$ is a QSD with rate of decay $\beta$. We get for all $t\geq 0$
\begin{equation*}
e^{-\beta t}=\mathbb{P}_{\mu_{\mbox{\tiny$\beta$}}}(\rT>t) = \int_{(0,\infty)}\!\!\!\mu_{\mbox{\tiny$\beta$}}(dr)e^{-r\,a_t}
\end{equation*}
Letting $t=\Phi(\lambda)$ for any $\lambda\geq 0$ we obtain
\begin{equation*}
e^{-\beta\Phi(\lambda)} = \int_{(0,\infty)}\mu_{\mbox{\tiny$\beta$}}(dr)e^{-r\lambda}
\end{equation*}
Consequently there is at most one QSD corresponding to the rate of decay $\beta$. Now suppose that the preceding formula defines a probability distribution on $(0,\infty)$ then the following calculation ensures that it is quasi-stationary:
\begin{eqnarray*}
\forall\lambda>0,\;\;\mathbb{E}_{\mu_{\mbox{\tiny$\beta$}}}\big[e^{-\lambda \mathrm{Z}_t}\, | \,\rT > t\big]\!\!&=&\!\!\frac{\mathbb{E}_{\mu_{\mbox{\tiny$\beta$}}}\big[e^{-\lambda \mathrm{Z}_t};\rT > t\big]}{\mathbb{P}_{\mu_{\mbox{\tiny$\beta$}}}(\rT > t)}=\frac{\mathbb{E}_{\mu_{\mbox{\tiny$\beta$}}}\big[e^{-\lambda \mathrm{Z}_t}\big]}{\mathbb{P}_{\mu_{\mbox{\tiny$\beta$}}}(\rT > t)}=\frac{\int_{(0,\infty)}\mu_{\mbox{\tiny$\beta$}}(dr)e^{-r\,u(t,\lambda)}}{e^{-\beta t}}\\
\!\!&=&\!\!e^{-\beta\big(\Phi(u(t,\lambda))-t\big)}=e^{-\beta\Phi(\lambda)}=\mathbb{E}_{\mu_{\mbox{\tiny$\beta$}}}[e^{-\lambda \rZ_0}]
\end{eqnarray*}
We now assume $\Psi(+\infty) \in (-\infty,0)$ and we prove that $\lambda\mapsto e^{-\beta\Phi(\lambda)}$ is indeed the Laplace transform of a probability measure $\mu_{\mbox{\tiny$\beta$}}$ on $(0,\infty)$. Let $x:=\beta/\nu(0,\infty)$, for all $\lambda > 0$ we have
\begin{equation*}
\mathbb{E}_x\big[e^{-\lambda \mathrm{Z}_t}\, | \,\rT > t\big] = \frac{\mathbb{E}_x\big[e^{-\lambda \mathrm{Z}_t};\rT > t\big]}{\mathbb{P}_x(\rT > t)} = \exp\Big(-x\big(u(t,\lambda)-a_t\big)\Big)
\end{equation*}
From (\ref{EqUt}) and the definition of $\Phi$ we get that
\begin{equation*}
\int_{u(t,\lambda)}^{a_t}\frac{du}{\Psi(u)}=\Phi(\lambda)
\end{equation*}
Using again (\ref{EqUt}) and the fact that $\Psi$ is non-positive, we get that $a_t \rightarrow \infty$ and $u(t,\lambda) \rightarrow \infty$ as $t\rightarrow\infty$. Since $\Psi(u)\rightarrow -\nu(0,\infty)$ as $u\rightarrow\infty$, one deduces that
\begin{equation*}
\int_{u(t,\lambda)}^{a_t}\frac{du}{\Psi(u)}\underset{t\rightarrow\infty}{\sim}\frac{u(t,\lambda)-a_t}{\nu(0,\infty)}
\end{equation*}
and therefore
\begin{equation*}
\mathbb{E}_x\big[e^{-\lambda \mathrm{Z}_t}|\rT > t\big] \underset{t\rightarrow\infty}{\longrightarrow} e^{-\Phi(\lambda)\, x\, \nu(0,\infty)}
\end{equation*}
Since $\Phi(\lambda)\rightarrow 0$ as $\lambda\downarrow 0$, we deduce that $\lambda\mapsto e^{-\Phi(\lambda)\,x\,\nu(0,\infty)}=e^{-\beta\Phi(\lambda)}$ is the Laplace transform of a probability measure on $[0,\infty)$. Moreover, it does not charge $0$ since $\Phi(\lambda)\rightarrow\infty$ as $\lambda\rightarrow\infty$.\\
We now suppose $\Psi(+\infty)=-\infty$. An easy adaptation of the preceding arguments ensures that for any $x,\lambda > 0$
\begin{equation*}
\mathbb{E}_x[e^{-\lambda \mathrm{Z}_t}|\rT > t] \underset{t\rightarrow\infty}{\longrightarrow} 0
\end{equation*}
Hence the limiting distribution is trivial: it is a Dirac mass at infinity. However, let us prove that $\lambda\mapsto e^{-\beta\Phi(\lambda)}$ is indeed the Laplace transform of a probability measure $\mu_{\mbox{\tiny$\beta$}}$ on $(0,\infty)$. For every $\epsilon > 0$, define the branching mechanism
\begin{equation*}
\Psi_{\epsilon}(u):=\int_{(0,\infty)}\!\!\!(e^{-hu}-1)(\tun_{\{h>\epsilon\}}\nu(dh)+\frac{1}{\epsilon}\delta_{-D\epsilon}(dh))=\frac{1}{\epsilon}(e^{D\epsilon u}-1) + \int_{(\epsilon,\infty)}\!\!\!(e^{-hu}-1)\,\nu(dh)
\end{equation*}
Observe that for any $u\geq 0$, $\Psi_{\epsilon}(u)\downarrow\Psi(u)$ as $\epsilon \downarrow 0$. Thus by monotone convergence we deduce that
\begin{equation*}
\forall \lambda\geq 0,\;\;\int_{\lambda}^{0}\frac{du}{\Psi_{\epsilon}(u)} \underset{\epsilon\downarrow 0}{\longrightarrow}\int_{\lambda}^{0}\frac{du}{\Psi(u)}
\end{equation*}
The first part of the proof applies to $\Psi_{\epsilon}$, and therefore the l.h.s. of the preceding equation is the Laplace exponent taken at $\lambda$ of an infinitely divisible distribution on $(0,\infty)$. Since the r.h.s. vanishes at $0$ and goes to $\infty$ at $\infty$, it is the Laplace exponent of an infinitely divisible distribution on $(0,\infty)$.\cqfd

\subsection{Proof of Theorem \ref{ThQDistrib}}
Fix $t \geq 0$. Since we are dealing with non-decreasing processes and since the asserted limiting process is continuous, the convergence of the finite-dimensional marginals suffices to prove the theorem (see for instance Th VI.3.37 in~\cite{JacodShiryaev}). By Proposition \ref{PropDt}, we know that $u^Q(t,\lambda)=\lambda e^{-Dt}$ is the function related to $\Psi^Q$ via (\ref{EqDefUt}). Hence we only need to prove that for all $n\geq 1$, all $n$-uplets $0\leq t_1 \leq \ldots \leq t_n \leq t$ and all coefficients $\lambda_1,\ldots,\lambda_n > 0$ we have
\begin{equation}\label{EquationLimitQProcess}
\lim\limits_{s\rightarrow\infty}-\frac{1}{x}\log\mathbb{E}_x[e^{-\lambda_1 \rZ_{t_1}-\ldots-\lambda_n \rZ_{t_n}}|\rT >t+s]=\lambda_1 d_{t_1}+\ldots+\lambda_n d_{t_n}
\end{equation}
Thanks to an easy recursion, we get
\begin{eqnarray*}
&&-\frac{1}{x}\log\mathbb{E}_x[e^{-\lambda_1 \rZ_{t_1}-\ldots-\lambda_n \rZ_{t_n}}|\rT >t+s]\\
\!\!\!&=&\!\!\! u\bigg(\!t_1,\lambda_1+u\Big(t_2\!-\!t_1,\lambda_2+\ldots+u\big(t_n\!-\!t_{n-1},\lambda_n+u(t+s-t_n,0+)\big)\ldots\Big)\!\bigg)\!-u(t+s,0+)
\end{eqnarray*}
To prove (\ref{EquationLimitQProcess}), we proceed via a recurrence on $n$. We check the case $n=1$. Recall that $u(t,\lambda)/\lambda\rightarrow d_t$ as $\lambda\rightarrow\infty$. Then the concavity of $\lambda \rightarrow u(t,\lambda)$ (that can be directly checked from (\ref{EqLKUt})) implies that $\partial_\lambda u(t,\lambda) \rightarrow d_t$ as $\lambda \rightarrow\infty$. Writing $u(t+s,0+)=u\big(t_1,u(t+s-t_1,0+)\big)$, the preceding arguments and the fact that $u(t+s-t_1,0+)=a_{t+s-t_1}\rightarrow\infty$ as $s\rightarrow\infty$ entail
\begin{equation*}
u\big(t_1,\lambda_1+u(t+s-t_1,0+)\big)\!-u(t+s,0+)\rightarrow \lambda_1d_{t_1}\; \mbox{ as }s\rightarrow\infty
\end{equation*}
Suppose now that the result holds at rank $n\!-\!1 \geq 1$, that is, (\ref{EquationLimitQProcess}) holds true for all $(n\!-\!1)$-uplets of times and coefficients. In particular
\begin{eqnarray*}
&&u\Big(t_2-t_1,\lambda_2+\ldots+u\big(t_n-t_{n-1},\lambda_n+u(t+s-t_n,0+)\big)\!\!\ldots\!\!\Big)-u(t+s-t_1,0+)\\
&&\!\!\underset{s\rightarrow\infty}{\sim}\!\! \lambda_2d_{t_2-t_1}+\ldots+\lambda_nd_{t_n-t_1}
\end{eqnarray*}
Therefore the argument of the case $n\!=\!1$ applies and shows that
\begin{eqnarray*}
&&\!\!\!\!u\bigg(t_1,\lambda_1+u\Big(t_2-t_1,\lambda_2+\ldots+u\big(t_n-t_{n-1},\lambda_n+u(t+s-t_n,0+)\big)\ldots\Big)\!\bigg)\!-u(t+s,0+)\\
&&\underset{s\rightarrow\infty}{\sim}\!\!\!\lambda_1d_{t_1} + \lambda_2d_{t_1}d_{t_2-t_1}+\ldots+\lambda_nd_{t_1}d_{t_n-t_1}
\end{eqnarray*}
which is the desired result since $d_{r+r'}=d_rd_{r'}$ for all $r,r'\geq 0$ by Proposition \ref{PropDt}.\cqfd

\subsection{Proof of Theorem \ref{ThRegularly}}
Recall the notation $a_t=u(t,0+)$ and that $a_t\rightarrow\infty$ as $t\rightarrow\infty$. Since $u\mapsto \Psi(u)/u$ is strictly increasing from $-\infty$ to $D$, there exists a positive function $f$ such that
\[ \Psi\big(f(t)^{-1}\big)f(t)\sim \Psi(a_t) \]
as $t\rightarrow\infty$. Since $\Psi(a_t)\rightarrow-\infty$ as $t\rightarrow\infty$, necessarily $f(t)\rightarrow\infty$. Fix $\lambda,x \in (0,\infty)$. For any $t\in (0,\infty)$, we have
$$ -\frac{1}{x}\log\bbE_x[e^{-\lambda \rZ_t/f(t)}\, | \, t < T] = u(t,\lambda\, f(t)^{-1})-a_t$$
We rely on two lemmas, whose proofs are postponed to the end of the subsection.
\begin{lemma}\label{LemmaPhi}
As $u\downarrow 0$, we have $\Phi(u) \sim u/(-\alpha\Psi(u))$.
\end{lemma}
Since $f(t) \rightarrow +\infty$ as $t\rightarrow\infty$ the lemma implies
$$\Psi(a_t)\Phi(\lambda\, f(t)^{-1}) \underset{t\rightarrow\infty}{\sim} -\frac{\Psi(a_t)\lambda}{\alpha f(t) \Psi(\lambda\, f(t)^{-1})}$$
Since $L$ is slowly varying at $0+$, we deduce that $\Psi(\lambda\, f(t)^{-1}) \sim \lambda^{1-\alpha}\Psi( f(t)^{-1})$ as $t\rightarrow\infty$. Thus the very definition of $f$ entails
\begin{equation}\label{EqPsiPhi}
\Psi(a_t)\Phi(\lambda\, f(t)^{-1}) \underset{t\rightarrow\infty}{\sim} -\lambda^{\alpha}\alpha^{-1}
\end{equation}
\vspace{-15pt}
\begin{lemma}\label{LemmaSimInt}
The following holds true as $t\rightarrow\infty$
$$\int_{u(t,\lambda\, f(t)^{-1})}^{a_t}\frac{dv}{\Psi(v)} \sim \int_{u(t,\lambda\, f(t)^{-1})}^{a_t}\frac{dv}{\Psi(a_t)}$$
\end{lemma}
From the latter lemma, we deduce
\begin{eqnarray*}
u(t,\lambda \, f(t)^{-1})-a_t &\underset{t\rightarrow\infty}{\sim}& -\Psi(a_t)\int_{u(t,\lambda\, f(t)^{-1})}^{a_t}\frac{dv}{\Psi(v)}= -\Psi(a_t)\Phi(\lambda\, f(t)^{-1})\\
&\underset{t\rightarrow\infty}{\sim}& \lambda^{\alpha}\alpha^{-1}
\end{eqnarray*}
where we use (\ref{EqPsiPhi}) at the second line. The theorem is proved.\cqfd\\
\textit{Proof of Lemma \ref{LemmaPhi}.} Recall the definition of $\Phi$. An integration by parts yields that for all $u\in[0,\infty)$
$$ \Phi(u) = -\frac{u}{\Psi(u)} + \int_u^0\frac{1}{\Psi(v)}\frac{v\Psi'(v)}{\Psi(v)}dv $$
Recall from Theorem 2 in~\cite{Lamperti58} that $v\Psi'(v)/\Psi(v) \rightarrow 1-\alpha$ as $v\downarrow 0$. Therefore an elementary calculation ends the proof.\cqfd\\
\textit{Proof of Lemma \ref{LemmaSimInt}.} For all $t\in[0,\infty)$, $a_t \leq u(t,\, \lambda\, f(t)^{-1})$. We write
$$ \int_{a_t}^{u(t,\, \lambda\, f(t)^{-1})}\frac{dv}{\Psi(v)}-\int_{a_t}^{u(t,\, \lambda\, f(t)^{-1})}\frac{dv}{\Psi(a_t)} = \int_{a_t}^{u(t,\, \lambda\, f(t)^{-1})}\frac{\Psi(a_t)-\Psi(v)}{\Psi(v)\Psi(a_t)}dv $$
The convexity of $\Psi$ implies that for all $v\in[a_t,u(t,\, \lambda\, f(t)^{-1})]$ we have
\begin{equation}\label{EqConvexity}
0 \leq \Psi(a_t)-\Psi(v) \leq -\Psi'(a_t)\big(v-a_t\big)
\end{equation}
Suppose that $t\mapsto u(t,\, \lambda\, f(t)^{-1})-a_t$ is bounded for large times. The fact that $\Psi'(v)/\Psi(v)$ goes to $0$ as $v\rightarrow\infty$ together with (\ref{EqConvexity}) then entail
\begin{eqnarray*}
0\leq\int_{a_t}^{u(t,\frac{\lambda}{f(t)})}\!\frac{\Psi(a_t)-\Psi(v)}{\Psi(v)\Psi(a_t)}dv&\leq&-\Big(u\Big(t,\frac{\lambda}{f(t)}\Big)-a_t\Big)\frac{\Psi'(a_t)}{\Psi(a_t)}\int_{a_t}^{u(t,\frac{\lambda}{f(t)})}\!\!\frac{dv}{\Psi(v)}\\
&\underset{t\rightarrow\infty}{\leq}& \!o\Big(\int_{a_t}^{u(t,\frac{\lambda}{f(t)})}\!\!\frac{dv}{\Psi(v)}\Big)
\end{eqnarray*}
which in turn proves the lemma. We are left with the proof of the boundedness of $t\mapsto u(t,\, \lambda\, f(t)^{-1})-a_t$ for large times. Fix $k \in (-D,\infty)$. Since $\Psi'(v)\uparrow D$ as $v\rightarrow\infty$, for $t$ large enough we get from (\ref{EqConvexity}) that $\Psi(v) \geq \Psi(a_t) - k(v-a_t)$ for all $v\in[\, a_t,u(t,\, \lambda\, f(t)^{-1})]$. A simple calculation then yields
$$ 0 \leq \frac{1}{k}\log\Big(1-k\frac{u(t,\, \lambda\, f(t)^{-1})-a_t}{\Psi(a_t)}\Big) \leq \int_{u(t,\lambda\, f(t)^{-1})}^{a_t}\frac{dv}{\Psi(v)} = \Phi(\lambda f(t)^{-1})$$
Using $\log(1+v) \geq v/2$ for $v$ small and since $\Phi\big(\lambda f(t)^{-1}\big) \rightarrow 0$, we get for $t$ large enough
$$ 0 \leq -\frac{u(t,\, \lambda\, f(t)^{-1})-a_t}{2\,\Psi(a_t)} \leq \Phi(\lambda f(t)^{-1}) $$
From (\ref{EqPsiPhi}), we deduce that $t\mapsto u(t,\, \lambda\, f(t)^{-1})-a_t$ is bounded for large times.\cqfd

\section{Proof of Proposition \ref{PropCritical}}
The proof is inspired by that of Theorem 1 in~\cite{Slack68} but for completeness we give all the details. Recall that $\Psi(u)=u^{1+\alpha}L(u)$ with $L$ slowly varying at $0$ and $\alpha \in (0,1)$ and that $\rT_0<\infty$ almost surely: consequently $q=0$ and (\ref{EqExtin}) holds true. Recall (\ref{EqUt}). We set for all $t\geq 0$, $v(t):=u(t,+\infty)$ which is finite by (\ref{EqExtin}). Observe that $v$ is decreasing from $+\infty$ to $0$. Grey p. 672~\cite{Grey74} proved that
\begin{equation}\label{EqProbaExtin}
\forall t\geq 0, x > 0,\;\; \bbP_x\big(t\geq \rT\big) = e^{-xv(t)}
\end{equation}
Since $\Psi(u)/u \rightarrow \infty$ as $u\rightarrow\infty$ we get for all $r > 0$
\begin{eqnarray*}
\frac{v(r)}{r\Psi(v(r))} &=& \frac{1}{r}\int_0^r\partial_s\Big( \frac{v(s)}{\Psi(v(s))}\Big)ds = \frac{1}{r}\int_0^r\partial_s v(s)\frac{\Psi\big(v(s)\big)-v(s)\Psi'\big(v(s)\big)}{\Psi\big(v(s)\big)^2}ds\\
&=& \frac{1}{r}\int_0^r\Big(\frac{v(s)\Psi'\big(v(s)\big)}{\Psi\big(v(s)\big)}-1\Big)ds
\end{eqnarray*}
where we use the identity $\partial_s v(s)=-\Psi\big(v(s)\big)$ at the second line. Since $\Psi$ is regularly varying at $0$, Theorem 2 in~\cite{Lamperti58} entails that $u\Psi'(u)/\Psi(u) \rightarrow 1+\alpha$ as $u\downarrow 0$. Taking the limit $r\rightarrow\infty$ in the above identity, one gets
\begin{equation}\label{EquationEquivalentCritical}
v(r)^\alpha L(v(r)) \sim \frac{1}{\alpha\, r}\; \mbox{ as }r\rightarrow\infty
\end{equation}
Since $v$ is a bijection from $(0,\infty)$ onto itself, for any $t\in(0,\infty)$ there exists a unique $s(t)=s \in(0,\infty)$ such that $v(s)=\lambda f(t)$. From the assumption $f(t) \sim v(t)$ as $t\rightarrow\infty$, we deduce that $s\rightarrow\infty$ as $t\rightarrow\infty$. We use (\ref{EquationEquivalentCritical}) and the slowness of the variation of $L$ to get as $t\rightarrow\infty$
$$ \frac{t}{s} \sim \frac{v(s)^\alpha L(v(s))}{v(t)^\alpha L(v(t))} \sim \frac{\lambda^\alpha f(t)^\alpha L(\lambda f(t))}{f(t)^\alpha L(f(t))} \sim \lambda^\alpha $$
Hence $\lambda^\alpha s \sim t$ as $t\rightarrow\infty$.
Using $\partial_r v(r)=-\Psi\big(v(r)\big)$ and (\ref{EquationEquivalentCritical}), we obtain for all $t > 0$
\begin{equation*}
\log\Big(\frac{v(t+s)}{v(t)}\Big) = \int_t^{t+s} \frac{\partial_r v(r)}{v(r)}dr =-\int_t^{t+s} v(r)^\alpha L\big(v(r)\big)dr \underset{t\rightarrow\infty}{\sim} -\frac{1}{\alpha}\log(1+\lambda^{-\alpha})
\end{equation*}
Using the above results, (\ref{EqProbaExtin}) and the identity $u(t,\lambda f(t))\!=\!u(t,v(s))\!=\!v(t+s)$ we get for all $t > 0$
\begin{eqnarray*}
\bbE_x\Big[e^{-\lambda Z_t f(t)}\, \big| \, t < \rT\Big] &=&\frac{\bbE_x\Big[e^{-\lambda Z_t f(t)}\Big] - \bbP_x(t \geq \rT)}{\bbP_x(t < \rT)} = \frac{e^{-x\, u(t,\lambda f(t))}-e^{-x\, v(t)}}{1-e^{-x\, v(t)}}\\
&\underset{t\rightarrow\infty}{\sim}& 1 - \frac{v(t+s)}{v(t)}\underset{t\rightarrow\infty}{\sim} 1 - (1+\lambda^{-\alpha})^{-1/\alpha}
\end{eqnarray*}
This ends the proof.\cqfd





\ACKNO{I would like to thank Amaury Lambert for a fruitful discussion, Thomas Duquesne for pointing out to me the paper of Slack, Olivier H\'enard for useful remarks and two referees for several comments that improved the presentation of this article. Part of this work was carried out while I was visiting Alison Etheridge at the Department of Statistics in Oxford: I am grateful to the Fondation sciences math\'ematiques de Paris for their funding.}


\end{document}